\documentclass[a4paper,11pt]{amsart}

\usepackage[top=3.5cm,bottom=3.5cm,left=3.6cm,right=3.6cm]{geometry}
\usepackage{amsfonts}
\usepackage{url}
\usepackage{amsmath}
\usepackage{amssymb}
\usepackage{amsthm}
\usepackage{graphicx}
\usepackage{hyperref}
\usepackage{enumerate}%pour changer la numérotation
\usepackage{mathtools}
\usepackage{varioref}
\usepackage[T1]{fontenc}
\usepackage{enumitem}

\newtheorem{dfn}{Définition}
\newtheorem{propo}[dfn]{Proposition}
\labelformat{propo}{proposition~#1}

\newtheorem{theo}[dfn]{Theorem}
\labelformat{theo}{theorem~#1}
\newtheorem{innercustomthm}{Theorem}
\newenvironment{customthm}[1]
  {\renewcommand\theinnercustomthm{#1}\innercustomthm}
  {\endinnercustomthm}

\labelformat{cor}{corollary~#1}
\newtheorem{lem}[dfn]{Lemma}
\labelformat{lem}{lemma~#1}
\theoremstyle{remark}

\newtheorem{rem}[dfn]{Remark}
\newtheorem{nota}[dfn]{Notation}

\newcommand{\suite}[2]{\left(#1\right)_{#2}}
\newcommand{\norm}[1]{\left\lVert #1 \right\rVert}
\newcommand{\ens}[1]{\left\{ #1\right\}}
\newcommand{\R}{\mathbb{R}}
\newcommand{\Z}{\mathbb{Z}}	
\newcommand{\N}{\mathbb{N}}

\newcommand{\ent}[1]{\lfloor #1\rfloor}

\newcommand{\f}{\mathcal{F}}

\newcommand{\abs}[1]{\left|#1\right|}

\providecommand{\newoperator}[3]{%
  \newcommand*{#1}{\mathop{#2}#3}}
\newoperator{\re}{\mathrm{Re}}{\,}
\newoperator{\im}{\mathrm{Im}}{\,}

\newcommand{\esper}[1]{\mathbb E\left[#1\right]}

\author{Davide Giraudo \and Dalibor Voln\'y}
\title{A counter example to central limit theorem in Hilbert spaces under a strong 
mixing condition}
\date{\today}

\begin{document}

%%%%%%%%%%%%%%%%%%%%%%%%%%%%%%%%%%%%%%%%%%%%%%%%%%%%%%%%%%%%%%%%%%%
%%                                                               %%
%% No need for \maketitle.                                       %%
%%                                                               %%
%%%%%%%%%%%%%%%%%%%%%%%%%%%%%%%%%%%%%%%%%%%%%%%%%%%%%%%%%%%%%%%%%%%

%%%%%%%%%%%%%%%%%%%%%%%%%%%%%%%%%%%%%%%%%%%%%%%%%%%%%%%%%%%%%%%%%%%
%%                                                               %%
%% Please replace what follows by the body of your article       %%
%% (up to the bibliography):                                     %%
%%                                                               %%
%%%%%%%%%%%%%%%%%%%%%%%%%%%%%%%%%%%%%%%%%%%%%%%%%%%%%%%%%%%%%%%%%%%

 \begin{abstract}
We show that in a separable infinite dimensional Hilbert space, uniform integrability 
of the square of the norm of normalized partial sums of a strictly stationary 
sequence, together with a strong mixing condition, does not 
guarantee the central limit theorem. 
 \end{abstract}

\maketitle



\maketitle

\section{Introduction and notations}
Let $(\Omega,\f,\mu)$ be a probability space and $(S,d)$ a separable metric space. We 
say that the sequence of random variables $\suite{X_n}{n\in\Z}$ from $\Omega$ to $S$ 
is \textit{strictly stationary} if for all integer $d$ and all integer $k$, the $d$-
uple $(X_1,\dots,X_d)$ has the same law as $(X_{k+1},\dots,X_{k+d})$. 

Rosenblatt introduced in \cite{MR0074711} the measure of dependence between two 
\\ sub-$\sigma$-algebras $\mathcal A$ and $\mathcal B$: 
\begin{equation*}
\alpha(\mathcal A,\mathcal B):=\sup\ens{\abs{\mu(A\cap B)-\mu(A)\mu(B)},A\in\mathcal 
A, B\in\mathcal B}. 
\end{equation*}
Another one is $\beta$-mixing, which is defined by 
\begin{equation*}
\beta(\mathcal A,\mathcal B):=\frac 12\sup\sum_{i=1}^I\sum_{j=1}^J\abs{\mu(A_i\cap 
B_j)-\mu(A_i)\mu(B_j)}, 
\end{equation*}
where the supremum is taken over the finite partitions $\ens{A_1,\dots,A_I}$ and 
$\ens{B_1,\dots,B_J}$ of $\Omega$, which consist respectively of elements of 
$\mathcal A$ and $\mathcal B$. It was introduced by Volkonskii and Rozanov in 
\cite{MR0105741}.

In order to measure dependence of a sequence of random variables, say $X:=\suite{X_j}
{j\in\Z}$ (assumed strictly stationary for simplicity), we define $\f_m^n$ as the 
$\sigma$-algebra generated by the $X_j$ for $m\leqslant j\leqslant n$, where 
$-\infty\leqslant m\leqslant n\leqslant +\infty$. 

Then mixing coefficients are defined by 
\begin{equation}\label{alpha_mix_def}
\alpha_X\left(n\right):=\alpha\left(\f_{-\infty}^0,\f_n^{+\infty}\right)
\end{equation}
\begin{equation}
\beta_X\left(n\right):=\beta\left(\f_{-\infty}^0,\f_n^{+\infty}\right),
\end{equation}
which will be simply writen $\alpha(n)$ (respectively $\beta(n)$) when there is no 
ambiguity. 

We say that the strictly stationary sequence $\suite{X_j}j$ is 
\textit{$\alpha$-mixing} (respectively \textit{$\beta$-mixing)} if $\lim_{n\to \infty}\alpha(n)=0$ 
(respectively $\lim_{n\to \infty}\beta(n)=0$). Sequences which are $\alpha$-mixing 
are also called \textit{strong-mixing}. Notice 
that the inequality $2\alpha(\mathcal A,\mathcal B)\leqslant
\beta(\mathcal A,\mathcal B)$ for any two sub-$\sigma$-algebras $\mathcal A$ and 
$\mathcal B$ implies that each $\beta$-mixing sequence is strong mixing. We refer the 
reader to Bradley's book \cite{MR2325294} for further information about mixing 
conditions. 

\bigskip 

Let $(V,\norm\cdot)$ be a separable normed space. We can represent a strictly 
stationary sequence $\suite{X_j}j$ by $X_j=f\circ T^j$, where $T\colon\Omega\to 
\Omega$ is measurable and measure preserving, that is, $\mu(T^{-1}(S))=\mu(S)$ for 
all $S\in\f$ (see \cite{MR0058896}, p.456, second paragraph). 

Given an integer $N$, we define $\displaystyle S_N(f):=\sum_{j=0}^{N-1}f\circ T^j$ and 
$(\sigma_N(f))^2:=\esper{\norm{S_N(f)}^2}$. 

When $V=\R^d$, $d\in\N^*$ it is well-known that if $\suite{f\circ T^j}{j\geqslant 0}$ 
satisfies the following assumptions:
\begin{enumerate}
\item\label{H1} $\lim_{N\to +\infty}\sigma_N(f)=+\infty$;
\item\label{H2} $\int fd\mu=0$:
\item\label{H3} $\lim_{n\to +\infty}\alpha(n)=0$;
\item\label{H4} the family $\ens{\frac{\norm{S_N(f)}^2}{(\sigma_N(f))^2},N\geqslant 
1}$ is uniformly integrable,
\end{enumerate}
then $\suite{\frac 1{\sigma_N(f)}S_N(f)}{N\geqslant 1}$ converges in distribution to 
a Gaussian law. It was established for $d=1$ by Denker \cite{MR899993}, Mori and Yoshihara 
\cite{MR886062}   using a 
blocking argument. Voln\'y \cite{MR936028} gave a proof for $d$ arbitrary based 
on approximation by an array of independent random variables. 

\bigskip 

A natural question would be: what if we replace $\R^d$ by another normed space? 

\bigskip 

First, we restrict ourselves to separable normed spaces in order to avoid 
measurability issues of sums of random variables. Corollary 10.9. in \cite{MR1102015} 
asserts that a separable Banach space $B$ with norm $\norm{\cdot}_{B}$ is isomorphic 
to a Hilbert space if and only if for all random variables $X$ with values in $B$, the 
conditions $\esper{\mathbf X}=0$ and $\esper{\norm{\mathbf X}_B^2}<\infty$ are 
necessary and sufficient for $X$ to satisfy the central limit theorem. By 
"$\mathbf X$ satisfies the CLT", we mean that if $\suite{\mathbf {X_j}}{j\geqslant 
1}$ is a sequence of independent random variables, with the same 
law as $X$, the sequence $\suite{n^{-1/2}\sum_{j=1}^n\mathbf{X_j}}{n\geqslant 1}$ 
weakly converges in $B$. Hence we cannot expect a generalization in 
a class larger than separable Hilbert spaces. Such a space is necessarily isomorphic 
to $\mathcal H:=\ell^2(\R)$, the space of square sumable sequences $\suite{x_n}
{n\geqslant 1}$ endowed with the inner product $\langle 
x,y\rangle_{\mathcal H}:=\sum_{n=1}^{+\infty}x_ny_n$. We shall denote by $
\mathbf{e_n}$ the sequence whose all terms are $0$, 
except the $n$-th which is $1$. Bold letters denote both randoms variables taking 
their values in $\mathcal H$ and elements of this space.

General considerations about probability measures and central limit theorem in Banach 
spaces are contained in Araujo and Gin\'e's book \cite{MR576407}.

\begin{nota}
 If $\suite{a_n}{n\geqslant 1}$, $\suite{b_n}{n\geqslant 1}$ are sequences of non-
 negative real numbers, 
$a_n\lesssim b_n$ means that $a_n\leqslant C b_n$, where $C$ does not depend on $n$. 
In an analogous way, we define $
a_n\gtrsim b_n$. When $a_n\lesssim b_n\lesssim a_n$, we simply write $a_n\asymp b_n$. 
\end{nota}

Our main results are 
\begin{customthm}{A}\label{thmA}
There exists a probability space $(\Omega,\f,\mu)$ such that given $0<q<1$, we can 
construct a strictly stationary sequence $\mathbf X=(\mathbf{f}\circ 
T^k)=\suite{\mathbf{X_k}}{k\in\N}$ defined on $\Omega$, taking its values in 
$\mathcal H$, such that:
\begin{enumerate}
\item[a)]\label{expect} $\esper{\mathbf{f}}=0$, $\esper{\norm{\mathbf{f}}^p_{\mathcal 
H}}$ is finite for each $p$;
\item[b)]\label{div_dev} the limit $\lim_{N\to \infty}\sigma_N(\mathbf{f})$ is infinite;
\item[c)]\label{mixing} the process $\suite{\mathbf{X_k}}{k\in\N}$ is $\beta$-mixing, 
more precisely, $\beta_{\mathbf X}(j)=O\left(\frac 1{j^q}\right)$; 
\item[d)]\label{UI} the family $\ens{\frac{\norm{S_N(\mathbf{f})}_{\mathcal H}^2}
{\sigma_N^2(\mathbf{f})},N\geqslant 1}$ is uniformly integrable;
\item[e)]\label{tightness} if $I\subset \N$ is infinite, the family $\ens{\frac{S_N(
\mathbf{f})}{\sigma_N(\mathbf{f})},N\in I}$ is not 
tight in $\mathcal H$; furthermore, given a sequence $\suite{c_N}{N\geqslant 1}$ of 
real numbers going to infinity, we have either
  \begin{itemize}
  \item  $\lim_{N\to +\infty}\frac{\sigma_N(\mathbf{f})}{c_N}=0$, hence 
  $\suite{\frac{S_N(\mathbf{f})}{c_N}}{N\geqslant 1}$ converges to $
\mathbf{0}_{\mathcal H}$ in distribution, or
  \item  $\limsup_{N\to +\infty}\frac{\sigma_N(\mathbf{f})}{c_N}>0$, and in this case 
  the collection $\ens{\frac{S_N(\mathbf{f})}{c_N},N\geqslant 1}$ is not tight. 
  \end{itemize}
\end{enumerate}
\end{customthm}

\begin{customthm}{A'}\label{thmA'}
Let $(b_N)_{N\geqslant 1}$ and $(h_N)_{N\geqslant 1}$ be sequences 
of positive real numbers, with $\lim_{N\to \infty}b_N=0$ and 
$\lim_{N\to\infty}h_N=\infty$. Then there exists a strictly stationary 
sequence $\mathbf X:=(\mathbf{f}\circ T^k)_{k\in\N}=(\mathbf{X_k})_{k\in\N}$ of 
random variables with values in $\mathcal H$ such that \ref{expect}, 
\ref{UI}, \ref{tightness} of Theorem~\ref{thmA} and the following two properties 
hold:
\begin{enumerate}
\item[b')] \label{div_dev'} 
we have $\sigma_N^2(\mathbf{f})\lesssim N\cdot h_N$ and 
$\frac{\sigma_N^2(\mathbf{f})}N\to\infty$;
\item[c')] \label{mixing'} 
the process $\suite{\mathbf{X_k}}{k\in\N}$ is $\beta$-mixing, and there is an 
increasing sequence $(n_k)_{k\geqslant 1}$ of integers such that for each $k$, 
$\beta_{\mathbf X}(n_k)\leqslant b_{n_k}$.
\end{enumerate}
\end{customthm}

\begin{rem}\label{remark_on_thm_MP06}
 Theorem~\ref{thmA} shows that Denker's result does not remain valid 
 in its full generality in the context of Hilbert space valued 
 random variables. 
 
 Furthermore, a careful analysis of the proof of Proposition~\ref{growth_of_variance} 
 shows that for the construction given in Theorem~\ref{thmA}, 
 we have $\sigma_N^2(f)=N\cdot h(N)$ with $h$ slowly varying in the 
 strong sense. Theorem~1 
 of \cite{MR2280514} does not remain valid in the Hilbert space setting. Indeed, the 
 arguments given in pages~654-655 show that the conditions of Denker's theorem 
 together with the assumption that $\sigma_N^2=N\cdot h(N)$ with $h$ slowly varying 
 in the strong sense imply those of Theorem~1. These arguments are still true 
 in the Hilbert space setting.

\end{rem}

\begin{rem}
 Theorem~\ref{thmA'} gives a control of the mixing coefficients on a subsequence. 
 When $b_N:=N^{-2}$ for example, the construction gives a better estimation for the 
 considered subsequence than what we get by Theorem~\ref{thmA}.
\end{rem}

Tone has established in \cite{MR2754401} a central limit theorem for strictly 
stationary random fields with values in $\mathcal H$ under $\rho'$-mixing conditions. 
For sequences, these coefficients are defined by 
\begin{equation*}
\rho'(n):=\sup\ens{\frac{\abs{\esper{\langle \mathbf{f},\mathbf{g}\rangle_{\mathcal 
H}}-\langle\esper{\mathbf{f}},\esper{\mathbf{g}}\rangle_{\mathcal 
H}}}{\norm{\mathbf{f}}_{\mathbb{L}^2(\mathcal 
H)}\norm{\mathbf{g}}_{\mathbb{L}^2(\mathcal H)}}},
\end{equation*}
where the supremum is taken over all the non-zero functions $\mathbf{f}$ and 
$\mathbf{g}$ such that $\mathbf{f}$ and $\mathbf{g}$ are respectively 
$\sigma(X_j,j\in S_1)$ and $\sigma(X_j,j\in S_2)$-measurable, where $S_1$ and $S_2$ 
are such that $\min_{s\in S_1,t\in S_2}\abs{s-t}\geqslant n$, while 
$\mathbb{L}^2(\mathcal H)$ denote the collection of equivalence classes of random 
variables $\mathbf{X}\colon \Omega\to \mathcal H$ such that 
$\norm{\mathbf{X}}_{\mathcal H}^2$ is integrable.

So "interlaced index sets" can be considered, which is not the case for $\alpha$ and 
$\beta$-mixing coefficient. 
Taking $f$ and $g$ as characteristic functions of elements of $\f_{-\infty}^0$ and 
$\f_n^{+\infty}$ respectively, one can see that $\alpha(n)\leqslant \rho'(n)$, hence 
$\rho'$-mixing condition is more restrictive than 
$\alpha$-mixing condition.

A partial generalization of the finite dimensional result was proved by Politis and 
Romano \cite{MR1309424}, namely, the conditions $\mathbb E\norm{\mathbf 
{X_1}}_{\mathcal H}^{2+\delta}$ finite for some positive $\delta$ and $\sum_j 
\alpha_{\mathbf X}(j)^{\frac{\delta}{2+\delta}}$ guarantees the convergence of 
$n^{-1/2}\sum_{j=1}^n\mathbf{X_j}$ to a Gaussian random 
variable $\mathbf{\mathcal N}$, whose covariance operator $S$ satisfies
\begin{equation*}
\esper{\langle \mathbf{\mathcal N},h\rangle^2}=\langle Sh,h\rangle_{\mathcal 
H}=\operatorname{Var}(\langle 
\mathbf{X_1},h\rangle)+2\sum_{i=1}^{+\infty}\operatorname{Cov}\left(\langle \mathbf 
{X_1},h\rangle,\langle\mathbf{X_{1+i}},h\rangle\right).
\end{equation*}

Similar results were obtained by Dehling \cite{MR705631}.

Rio's inequality \cite{MR1251142} asserts that given two real valued random variables 
$X$ and $Y$ with finite two order moments, 
\begin{equation*}
\abs{\esper{XY}-\esper{X}\esper{Y} }\leqslant 2\int_0^{\alpha\left(\sigma(X),
\sigma(Y)\right)}Q_X(u)Q_Y(u)\mathrm du.
\end{equation*} 
It was extented by Merlev\`ede et al. \cite{MR1468399}, namely, if $\mathbf X$ and 
$\mathbf Y$ are two random variables with values in $\mathcal H$, with respective 
quantile function $Q_{\norm{\mathbf X}_{\mathcal H}}$ and $Q_{\norm{\mathbf 
Y}_{\mathcal H}}$, then 
\begin{equation*}
\abs{\esper{\langle\mathbf X,\mathbf Y\rangle_{\mathcal H}}-\langle \esper{\mathbf 
X},\esper{\mathbf Y}\rangle_{\mathcal H}}\leqslant 
18\int_0^{\alpha} Q_{\norm{\mathbf X}_{\mathcal H}}Q_{\norm{\mathbf Y}_{\mathcal 
H}}\mathrm du,
\end{equation*}
where $\alpha:=\alpha(\sigma(\mathbf X),\sigma(\mathbf Y))$.

From this inequality, they deduce a central limit theorem for a stationary sequence 
$\suite{\mathbf{X_j}}{j\in\Z}$ of $\mathcal H$-valued 
zero-mean random variables satisfying 
\begin{equation}\label{DMR_cond}
\int_0^1\alpha^{-1}(u)Q^2_{\norm{\mathbf{X_0}}_{\mathcal H}}(u)\mathrm du<\infty,
\end{equation}
where $\alpha^{-1}$ is the inverse function of $x\mapsto \alpha_{\mathbf X}(\ent x)$.

Discussion after Corollary 1.2 in \cite{MR2117923} proves that the later result 
implies Politis' one.

Relative optimality of condition \eqref{DMR_cond} (cf. \cite{MR1262892}) can give a 
finite-dimensional counter-example to the central limit theorem when this condition 
is not satisfied. Here, the condition of uniform integrability prevents such 
counter-examples.

Defining $\alpha_{2,\mathbf X}(n):=\sup_{i\geqslant j\geqslant 
n}\alpha(\f_{-\infty}^0,\sigma(\mathbf{X_i},\mathbf{X_j}))$ and $Q_{X_0}$ the 
right-continuous inverse of the function $t\mapsto \mu\ens{\norm {\mathbf 
{X_0}}_{\mathcal H}>t}$ (that is, 

$Q_{\mathbf{X_0}}(u):=\inf\ens{t\in\R,
\mu\ens{\norm{\mathbf{X_0}}_{\mathcal H}>t}\leqslant u}$), Dedecker and Merlev\`ede 
have shown in \cite{MR2731073} that under the assumption 
\begin{equation*}
\sum_{k=1}^{+\infty}\int_0^{\alpha_{2,\mathbf X}(k)}Q_{\mathbf {X_0}}^2(u)\mathrm 
du<\infty,
\end{equation*}
we can find a sequence $\suite{\mathbf{Z_i}}{i\in\N}$ of Gaussian random variables 
with values in $\mathcal H$ such that almost surely, 
\begin{equation*}
\norm{\mathbf{S_n}-\sum_{i=1}^n\mathbf{Z_i}}_{\mathcal H}=o\left(\sqrt{n\log\log 
n}\right).
\end{equation*}
\section{The proof}

\subsection{Construction of $f$}

In order to construct a counter-example, we shall need the following lemma, which 
will be proved later. 

We will denote $U$ the Koopman operator associated to $T$, which acts on measurable 
functions by $U(f)(x):=f(T(x))$.

\begin{lem}\label{exist_syst}
 Let $\suite{u_k}{k\geqslant 1}\subset (0,1)$ be a sequence of numbers. Then there 
 exists a dynamical system $(\Omega,\f,\mu,T)$ and a sequence of random variables 
 $\suite{\xi_k}{k\geqslant 1}$ such that 
\begin{enumerate}
 \item for each $k\geqslant 1$, $\mu(\xi_k=1)=\mu(\xi_k=-1)=\frac{u_k}2$ and 
 $\mu(\xi_k=0)=1-u_k$;
 \item the random variables $(U^i\xi_k,k\geqslant 1,i\in \Z)$ are mutually 
 independent.
\end{enumerate}

\end{lem}

Recall that $\mathbf{e_k}$ is the $k$-th element of the canonical orthonormal system 
of $\mathcal H=\ell^2(\R)$. We define 
\begin{equation}
f_k:=\sum_{i=0}^{n_k-1}U^{-i}\xi_k\mbox{ and 
}\mathbf{f}:=\sum_{k=1}^{+\infty}f_k\mathbf{e_k},
\end{equation}
where the $\xi_i$'s are constructed using to Lemma~\ref{exist_syst} taking 
$u_k:=n_k^{-2}$.
Conditions on the increasing sequence of integers $\suite{n_k}{k\geqslant 1}$ will be 
specified latter.

Then $\mathbf{X_k}:=\mathbf{f}\circ T^k$ is a strictly stationary sequence. 
Note that $\norm{\mathbf f}_{\mathcal H}^2$ is an integrable random variable whenever 
$\sum_k\frac 1{n_k}$ is convergent. In the sequel, the choice of $n_k$ will guarantee this condition. 

\subsection{Preliminary results}

We express $S_N(f_k)$ as a linear combination of independent random variables. 
By direct computations, we get   
\begin{equation}\label{expre_f_k}
f_k=n_k\xi_k+(I-U)\sum_{i=1-n_k}^{-1}(n_k+i)U^{i}\xi_k,
\end{equation}
hence 
\begin{align*}
S_N(f_k)&=n_k\sum_{j=0}^{N-1}U^j\xi_k+ \sum_{i=1-n_k}^{-1}(n_k+i)U^{i}\xi_k-
\sum_{i=N-n_k+1}^{N-1}(n_k+i-N)U^i\xi_k.
\end{align*}
This formula can be simplified if we distinguish the cases $N\geqslant n_k$ and 
$n_k<N$ (we break the third sum at the index $i=0$ if necessary). This gives 
\begin{multline}\label{sum_expre_N<n_k}
 S_N(f_k)=\sum_{j=0}^{N-1}(N-j)U^j\xi_k+\sum_{j=1-n_k}^{N-n_k}(n_k+j)U^j\xi_k\\
+N\sum_{j=1+N-n_k}^{-1}U^j\xi_k, \quad \mbox{if } N<n_k,
\end{multline}
\begin{multline}\label{sum_expre_N>>n_k}
 S_N(f_k)=n_k\sum_{j=0}^{N-n_k}U^j\xi_k+\sum_{j=N-n_k+1}^{N-1}(N-j)U^j\xi_k\\
+\sum_{j=1-n_k}^{-1}(n_k+j)U^j\xi_k, \quad\mbox{ if }N\geqslant n_k.
\end{multline}

%\label{sum_expr}

The computation of the expectation of the square of partial sums gives 
\begin{equation}\label{standard_dev}
\sigma_N^2(f_k)=\begin{cases}
\frac 1{n_k^2}\displaystyle\left(2\sum_{j=1}^Nj^2+(n_k-N-1)N^2\right) &\mbox{ if 
}N<n_k,\\
\frac 1{n_k^2}\displaystyle\left(n_k^2(N-n_k+1)+2\sum_{j=1}^{n_k-1}j^2\right)&\mbox{ 
if }N\geqslant n_k.
\end{cases}
\end{equation}

\begin{nota}
If $N$ is a positive integer and $\suite{n_k}{k\geqslant 1}$ is an increasing 
sequence of integers, denote by $i(N)$ the unique integer for which 
$n_{i(N)}\leqslant N<n_{i(N)+1}$. 
\end{nota}

\begin{propo}\label{growth_of_variance}
 Assume that $(n_k)_{k\geqslant 1}$ satisfies the condition 
\begin{equation}\tag{C}\label{lacunarity}
 \mbox{there is }p>1\mbox{ such that for each }k,\quad n_{k+1}\geqslant n_k^p. 
\end{equation}

 Then $\sigma^2_N(\mathbf{f})\asymp N\cdot i(N)$. 
\end{propo}

\begin{proof}
Using \eqref{standard_dev}, the fact that $M^3\asymp \sum_{j=1}^Mj^2$ and 
$\sigma_N^2(\mathbf f)=\sum_{k\geqslant 1}\sigma_N^2(f_k)$, we have 
\begin{equation}\label{expl_sigma}
\sigma_N^2(\mathbf f)\geqslant \sum_{k=1}^{i(N)}\sigma_N^2(f_k)\asymp 
N\sum_{j=1}^{i(N)}1=N\cdot i(N).
\end{equation}
From \eqref{standard_dev} in the case $n_k\geqslant N$, we deduce
\begin{equation}
 \sum_{k\geqslant i(N)+1}\sigma_N^2(f_k)\lesssim \sum_{k\geqslant i(N)+1}\frac{N^2}
 {n_k}\leqslant \frac{N^2}{n_{i(N)+1}}+\sum_{k\geqslant i(N)+1}\frac{N^2}{n_k}\frac 
 1{n_k^{p-1}}.
\end{equation}
Since $n_{i(N)+1}\geqslant N$ and the series $\sum_{k\geqslant 1}n_k^{1-p}$ is 
convergent (by the ratio test), we obtain 
\begin{equation}\label{bound_sigma}
 \sum_{k\geqslant i(N)+1}\sigma_N^2(f_k)\lesssim N+N\sum_{k\geqslant i(N)+1}\frac 
 1{n_k^{p-1}}\lesssim N.
\end{equation} 
Combining \eqref{expl_sigma} and \eqref{bound_sigma}, we get 
\begin{equation}
N\cdot i(N)\lesssim \sigma_N^2(\mathbf f)\lesssim 
\sum_{k=1}^{i(N)}\sigma_N^2(f_k)+\sum_{k\geqslant 
i(N)+1}\sigma_N^2(f_k)\lesssim N\cdot i(N)+N\lesssim N\cdot i(N).
\end{equation}
\end{proof}

\begin{propo}\label{condition_moments}
 Assume that $\sum_kn_k^{-a}$ is convergent for any positive real number $a$. Then 
 for each integer $p$, $\norm{\mathbf f}_{\mathcal H}$ has a finite moment of order 
 $p$.
\end{propo}
\begin{proof}
We shall use Rosenthal's inequality (Theorem~3, \cite{MR0271721}): there exists a 
constant $C$ depending only on $q$ such that if $M$ is an 
integer, $Y_1,\dots,Y_M$ are independent real valued zero mean random variables for 
which $\mathbb E\abs{Y_i}^q<\infty$ for each $i$, then   
\begin{equation}\label{Rosenthal}
\mathbb E\abs{\sum_{j=1}^MY_j}^q\leqslant 
C\left(\sum_{j=1}^M\mathbb  
E\abs{Y_j}^q+\left(\sum_{j=1}^M\esper{Y_j^2}\right)^{q/2}\right).
\end{equation}
If $q=2p$ is given then we have 
\begin{equation}
 \mathbb E\abs{f_k}^{2p}\lesssim n_k^{-1}+n_k^{-p}\lesssim n_k^{-1}.
\end{equation}
\end{proof}

We provide a sufficient condition for the uniform integrability of the family 
$\mathcal S:=\ens{\frac{\norm{S_N(\mathbf f)}_{\mathcal H}^2}{\sigma^2_N(
\mathbf f)},N\geqslant 1}$.

\begin{propo}\label{condition_for_UI}
If $(n_k)_{k\geqslant 1}$ satisfies \eqref{lacunarity}, then $\mathcal S$ is 
uniformly integrable. 
\end{propo}
\begin{proof}
Let $i_0$ be such that for each integer $k$, $n_{k+i_0}\geqslant n_k^2$.
For $N\geqslant 1$, we have:
\begin{equation*}
\frac{\norm{S_N(\mathbf f)}_{\mathcal H}^2}{\sigma^2_N(\mathbf f)}=
\sum_{j=1}^{i(N)-1}\frac{\abs{S_N(f_j)}^2}{\sigma^2_N(\mathbf f)}+
\frac{\abs{S_N(f_{i(N)})}^2}{\sigma^2_N(\mathbf f)}+\frac{\abs{S_N(f_{i(N)+1})}^2}
{\sigma^2_N(\mathbf f)}+\sum_{j\geqslant 
i(N)+2}\frac{\abs{S_N(f_j)}^2}{\sigma^2_N(\mathbf f)},
\end{equation*}
Since $\mathbb E\left[\sum_{j=i-i_0}^{i(N)-1}\frac{\abs{S_N(f_j)}^2}{\sigma^2_N(\mathbf f)}
\right]\leqslant 
\frac{i_0}{i(N)}$ 
it is enough to prove that the families 
\begin{align*}
\mathcal S_1&:=\ens{\sum_{k=1}^{i(N)-i_0}\frac{\abs{S_N(f_k)}^2}{\sigma^2_N(\mathbf 
f)},N\geqslant 1},\\
\mathcal S_2&:=\ens{\frac{\abs{S_N(f_{i(N)})}^2}{\sigma^2_N(\mathbf f)},N\geqslant 
1}=:\ens{u_N,N\geqslant 1},\\
\mathcal S_3&:=\ens{\frac{\abs{S_N(f_{i(N)+1})}^2}{\sigma^2_N(\mathbf f)},N\geqslant 
1}=:\ens{v_N,N\geqslant 1}, \mbox{ and }\\
\mathcal S_4&:=\ens{\sum_{k\geqslant i(N)+2}\frac{\abs{S_N(f_k)}^2}
{\sigma^2_N(\mathbf f)},N\geqslant 1}
\end{align*}
are uniformly integrable. For $\mathcal S_1$ and $\mathcal S_4$, we shall show that 
these families are bounded in 
$\mathbb L^p$ for $p\in (1,2]$ as in \eqref{lacunarity}. 

\begin{itemize}
\item for $\mathcal S_1$: using the expression in \eqref{sum_expre_N>>n_k}  and 
\eqref{Rosenthal} with $q:=2p>2$, we have
\begin{align*}
\esper{\abs{S_N(f_k)}^{2p}}&\leqslant 
C\left(2\sum_{j=1}^{n_k}\frac{j^{2p}}{n_k^2}+\frac{n_k^{2p}(N-n_k)}{n_k^2}\right)+
C\left(2\sum_{j=1}^{n_k}\frac{j^2}{n_k^2}+\frac{(N-n_k)n_k^2}{n_k^2}\right)^p\\
&\lesssim \frac 1{n_k^2}\left(n_k^{2p+1}+(N-n_k)n_k^{2p}\right)+\frac 
1{n_k^{2p}}\left(n_k^3+(N-n_k)n_k^2\right)^p\\
&=\frac{Nn_k^{2p}}{n_k^2}+\frac{N^pn_k^{2p}}{n_k^{2p}}\\
&=Nn_k^{2(p-1)}+N^p
\end{align*}
hence 
\begin{equation*}
\norm{S_N(f_k)^2}_{p}\lesssim N^{1/p}n_k^{2\frac{p-1}p}+N, 
\end{equation*}
which gives 
\begin{align*}
\norm{\sum_{k=1}^{i(N)-i_0}\frac{\abs{S_N(f_k)}^2}{\sigma^2_N(\mathbf f)}}_p
&\lesssim \frac{\sum_{k=1}^{i(N)-i_0}(N^{1/p}n_k^{2\frac{p-1}p}+N)}{\sigma_N^2(\mathbf 
f)}\\
&\lesssim \frac{N^{1/p}i(N)n_{i(N)-i_0}^{2\frac{p-1}p}+Ni(N)}{\sigma_N^2(\mathbf f)}\\
&\lesssim \left(\frac{n_{i(N)-i_0}^2}{N}\right)^{1-1/p}+1
\end{align*}

From \eqref{expl_sigma}, we get 
\begin{equation*}
\norm{\sum_{k=1}^{i(N)-1}\frac{\abs{S_N(f_k)}^2}{\sigma^2_N(\mathbf f)}}_{p}\lesssim  
\frac{n_{i(N)}^{2\frac{p-1}p}}{n_{i(N)}}+1=n_{i(N)}^{\frac{p-2}p}+1.
\end{equation*}
Since $p-2\leqslant 0$, we obtain that $\mathcal S_1$ is bounded in $\mathbb L^p$ 
hence uniformly integrable. 

\item for $\mathcal S_2$: using \eqref{sum_expre_N>>n_k} in the case $n_k\leqslant N$ 
and Proposition~\ref{growth_of_variance}, we get 
\begin{equation}
\norm{u_N}_1\lesssim \frac{N}{\sigma_N^2(\mathbf f)}\lesssim \frac 1{i(N)}.
\end{equation}
 Since $\norm{u_N}_1\to 0$ and $u_N\in\mathbb L^1$ for each $N$, the family 
 $\mathcal S_2$ is uniformly integrable.  

\item for $\mathcal S_3$: using \eqref{sum_expre_N<n_k} in the case $n_k> N$ and 
Proposition~\ref{growth_of_variance}, we get 
\begin{equation}
\norm{v_N}_1\lesssim \frac{N^2}{n_{i(N)+1}\sigma_N^2(\mathbf f)}\lesssim \frac{N}
{N\cdot i(N)}.
\end{equation}
 Since $\norm{v_N}_1\to 0$ and $v_N\in\mathbb L^1$ for each $N$, the family 
 $\mathcal S_3$ is uniformly integrable.  

\item for $\mathcal S_4$: as for $\mathcal S_1$, we shall show that this family is 
bounded in $\mathbb{L}^p$ with $p\in (1,2]$. We have, using \eqref{sum_expre_N<n_k} and 
\eqref{Rosenthal} 
\begin{align*}
\esper{\abs{S_N(f_k)}^{2p}}&\lesssim \frac 1{n_k^2}(N^{2p+1}+N^{2p}(n_k-N))+\frac 
1{n_k^{2p}}(N^3+(n_k-N)N^2)^{p}\\
&=\frac{N^{2p}}{n_k}+\frac{N^{2p}}{n_k^p}\\
&\lesssim \frac{N^{2p}}{n_k}
\end{align*}
as $N\leqslant n_k$. We thus get that 
\begin{equation*}
\norm{\sum_{k\geqslant i(N)+2}\abs{S_N(f_k)}^2}_{p}
\lesssim  N^2\sum_{k\geqslant i(N)+2}\frac 1{n_k^{1/p}}.
\end{equation*}
Also, using \eqref{standard_dev}, we have 
\begin{equation*}
\sigma_N^2(\mathbf f)\gtrsim N^2\sum_{k\geqslant i(N)+1}\frac 1{n_k}.
\end{equation*}
The condition $n_{k+1}\geqslant n_k^p$ gives boundedness in $\mathbb L^p$ of 
$\mathcal S_4$.
\end{itemize}
This concludes the proof of \ref{UI}.
\end{proof}

\begin{propo}\label{condition_for_non_tightness}
 Assume that $(n_k)_{k\geqslant 1}$ is such that $\mathcal S$ is uniformly integrable 
 and $\sum_kn_k^{-1}$ is convergent. Then for each $I\subset\N$ infinite, the 
 collection $\ens{\frac{S_N(\mathbf f)}{\sigma_N(\mathbf f)},N\in I}$ is
not tight in $\mathcal H$. Its finite-dimensional distributions converge to $0$ in 
probability. 

Furthermore, if $(c_N)_{N\geqslant 0}$ is a sequence of positive numbers going to 
infinity, we have either
  \begin{itemize}
  \item  $\lim_{N\to +\infty}\frac{\sigma_N(\mathbf f)}{c_N}=0$, hence 
  $\suite{\frac{S_N(\mathbf f)}{c_N}}{N\geqslant 1}$ converges to 
$\mathbf 0_{\mathcal H}$ in distribution, or
  \item  $\limsup_{N\to +\infty}\frac{\sigma_N(\mathbf f)}{c_N}>0$, and in this case 
  the sequence 
$\ens{\frac{S_N(\mathbf f)}{c_N},N\geqslant 1}$ is not tight. 
  \end{itemize} 
\end{propo}

\begin{proof} We first prove that the finite dimensional distributions of 
$\frac{S_N(\mathbf f)}{\sigma_N(\mathbf f)}$ converge weakly to $0$. 

For each $d\in\N$, we have 
 $\frac{\langle S_N(\mathbf f),\mathbf{e_d}\rangle_{\mathcal H}}{\sigma_N(\mathbf 
 f)}\to 0$ in distribution. 
Indeed, we have by \eqref{expre_f_k} that $\langle S_N(\mathbf f),
\mathbf{e_d}\rangle_{\mathcal H} =n_d\sum_{i=0}^{N-1}U^i\xi_d+(I-U^N)\sum_{i=1-
n_d}^{-1}(n_d+i)U^i\xi_d$. 
We conclude noticing that  
 $\sigma_N(\mathbf f)^{-1}(I-U^N)\sum_{i=1-n_d}^{-1}(n_d+i)U^i\xi_d$ goes to $0$ in 
 probability as $N$ goes to infinity, using Proposition~\ref{growth_of_variance} and 
 the estimate
 \begin{equation*}
 \mathbb  E\left(n_d\sum_{i=0}^{N-1}U^i\xi_d\right)^2=N\lesssim \frac{\sigma_N^2(
 \mathbf f)}{i(N)}
 \end{equation*} 

This can be extended replacing $\mathbf{e_d}$ by any $\mathbf{v}\in\mathcal H$ by an 
application of Theorem 4.2. in \cite{MR0233396}. By Proposition 4.15 
in \cite{MR576407}, the only possible limit is the Dirac measure at $\mathbf 
0_{\mathcal H}$.

Assume that the sequence $\ens{\frac{S_N(\mathbf f)}{\sigma_N(
\mathbf f)},N\geqslant 1}$ is tight. 
The sequence $\left(\frac{\norm{S_N(\mathbf{f})}_{\mathcal H}^2}
{\sigma^2_N(\mathbf{f})}\right)_{N\geqslant 1}$ is a uniformly integrable sequence 
of random variables of mean $1$. A weakly convergent subsequence would go to 
$\mathbf{0}_{\mathcal H}$. According to Theorem 5.4 in 
\cite{MR0233396}, we should have that the limit random variable has expectation $1$. 
This contradiction gives the result when $I=\N\setminus\ens 0$. Applying this 
reasonning to subsequences, one can see that for any infinite 
subset $I$ of $\N\setminus\ens 0$, the family $\ens{\frac{S_N(\mathbf f)}
{\sigma_N(\mathbf f)},N\in I}$ is not tight. 

Let $\suite{c_N}{N\geqslant 1}$ be a sequence of positive real numbers such that 
$\lim_{N\to +\infty}c_N=+\infty$. 
\begin{itemize}
\item first case: $\frac{\sigma_N(\mathbf f)}{c_N}$ converges to $0$. In this case, 
the sequence $\left(\frac{\norm{S_N(\mathbf f)}^2}{c_N^2}\right)_{N\geqslant 1}$ 
converges to $0$ in $\mathbb L^1$, hence the sequence 
$\left(\frac{S_N(\mathbf f)}{c_N}\right)_{N\geqslant 1}$ converges in distribution to 
$\mathbf 0_{\mathcal H}$.
\item second case: $\limsup_{N\to \infty}\frac{\sigma_N(\mathbf f)}{c_N}>0$. Hence 
there is some $r>0$ and a sequence of integers $l_i\uparrow\infty$ such that for each 
$i$, $\frac{\sigma_{l_i}(\mathbf f)}{c_{l_i}}\geqslant \frac 1r$, that is, 
$c_{l_i}\leqslant r\sigma_{l_i}(\mathbf f)$. 

Assume that the family $\ens{\frac{S_{l_i}(\mathbf f)}{c_{l_i}},i\geqslant 1}$ is 
tight. This means that given a positive 
$\varepsilon$, one can find a compact set $K=K(\varepsilon)$ such that for each $i$, 
$\mu\ens{\frac{S_{l_i}(\mathbf f)}{c_{l_i}}\in K}>1-\varepsilon$. We can assume that 
this compact set is convex and contains $0$ (we consider the closed convex hull of 
$K\cup\ens 0$, which is compact by Theorem 5.35 in \cite{MR2378491}). 
Then we have 
\begin{align*}
\ens{\frac{S_{l_i}(\mathbf f)}{c_{l_i}}\in K}
&=\ens{\frac{S_{l_i}(\mathbf f)}{\sigma_{l_i}(\mathbf f)}\in \frac{c_{l_i}}
{\sigma_{l_i}(\mathbf f)}K}\\
&\subset \ens{\frac{S_{l_i}(\mathbf f)}{\sigma_{l_i}(\mathbf f)}\in rK},
\end{align*}
and we would deduce tightness of $\ens{\frac{S_{l_i}(\mathbf f)}{\sigma_{l_i}
(\mathbf f)},i\geqslant 1}$, which cannot happen.

\end{itemize}
\begin{rem}
In the second case, it may happen that the finite dimensional distributions does not 
converge to degenerate ones, for example with $c_N:=N$.
\end{rem}
\end{proof}

\subsection{Proof of Theorem~\ref{thmA}}

Notice that if $n_{k+1}\geqslant n_k^p$ for some $p>1$ and $n_1=2$, then 
$n_k\geqslant 2^{p^k}$, hence the condition 
of Proposition~\ref{condition_moments} is fulfilled. We get \ref{expect} since each 
$f_k$ has expectation $0$.

We denote $\lfloor x\rfloor:=\sup\ens{k\in\Z, k\leqslant x}$ the integer part of the 
real number $x$.

\begin{propo}
Let $p>1$.
With $n_k:=\lfloor 2^{p^k}\rfloor$ (which satisfies \eqref{lacunarity}), we have for 
each positive integer $l$,
\begin{equation*}
\beta_{\mathbf X}(l)\lesssim \frac 1{l^{\frac 1p}}.
\end{equation*}
\end{propo}
\begin{proof}
We define $\beta_k(n)$ as the $n$-th $\beta$-mixing coefficient of the sequence 
$(f_k\circ T^i)_{i\geqslant 0}$. 

By Lemma~5 of \cite{Giraudo20143769}, we have the estimate $\beta_k(0)\leqslant 4n_k^{-1}$ for each 
$k$. Using then Proposition~4
of this paper (cf. \cite{MR2325294} for a proof), we get that 
$\beta_{\mathbf X}(n_k)\lesssim\sum_{j\geqslant k}\frac 1{n_j}$ for each integer 
$k$. Since $p^i\geqslant i$ for $i$ large enough,
\begin{equation*}
\sum_{j\geqslant k}\frac 1{n_j}=\sum_{i=0}^{+\infty}\frac 1{2^{p^{i+k}}}=
\sum_{i=0}^{+\infty}\frac 1{2^{p^ip^k}}\lesssim \sum_{i=0}^{+\infty}\frac 1{2^i}\frac 
1{2^{p^k}}=\frac 2{2^{p^k}},
\end{equation*}
we get
\begin{equation*}
\beta_{\mathbf X}(N)\leqslant \beta_X(n_{i(N)})\lesssim \frac 1{n_{i(N)}}=\frac 
1{n_{i(N)+1}^{1/p}}\leqslant \frac 1{N^{1/p}}.
\end{equation*} 
\end{proof}
This proves \ref{mixing}. For any $p$, the choice $n_k:=\lfloor 2^{p^k}\rfloor$ 
satisfies the condition of Proposition~\ref{condition_for_UI}, which proves \ref{UI}. 
We conclude the proof by Proposition~\ref{condition_for_non_tightness}.

\begin{rem}
For each of these choices, $\sigma^2_N(\mathbf f)$ behaves asymptotically like 
$N\log\log N$. Theorem~\ref{thmA'} shows that 
we can construct a process which satisfies the same asymptotic behavior of partial 
sums and has a variance close to a linear one.

A question would be: can we construct a strictly stationary sequence with all the
properties of Theorem~\ref{thmA}, except \ref{div_dev} which is replaced by an 
assumption of linear variance?
\end{rem}

\subsection{Proof of Theorem~\ref{thmA'}}

Let $\suite{h_N}{N\geqslant 1}$ be the sequence involved in Theorem~\ref{thmA'}. We 
define for an integer $u$ the quantity $h^{-1}(u):=\inf\ens{j\in\N,h_j\geqslant u}$. 

If $(b_k)_{k\geqslant 1}$ is the given sequence (that can be assumed decreasing), we 
define inductively 
\begin{equation}
 n_{k+1}:=\max\ens{n_k^2,\lfloor\frac{2^k}{b_{n_k}}\rfloor,h^{-1}(k)}.
\end{equation}
Let $N$ be an integer. We assume without loss of generality that the growth of the 
sequence $(h_N)_{N\geqslant 1}$ is slow enough in order to guarantee that there 
exists $k$ such that $N=h^{-1}(k)$. We then have $i(N)\leqslant 
k+1\leqslant h_N+1$, hence using Proposition~\ref{growth_of_variance}, we get b'). 

We have $n_k\geqslant 2^{2^k}$ hence by a similar argument as in the proof of 
Theorem~\ref{thmA}, \ref{expect} is satisfied. 

By a similar argument as in \cite{Giraudo20143769}, we get $\beta_{
\mathbf X}(n_k)\leqslant b_{n_k}$, hence c') holds. 

\begin{rem}
By \eqref{DMR_cond}, we cannot expect the relationship $\beta_{\mathbf X}
(\cdot)\leqslant b_{\cdot}$ for the whole sequence. 
\end{rem}

Since for each $k$, $n_{k+1}\geqslant n_k^2$, Proposition~\ref{condition_for_UI} 
and~\ref{condition_for_non_tightness} apply. This concludes the proof of 
Theorem~\ref{thmA'}. 

\begin{proof}[Proof of Lemma~\ref{exist_syst}]
Let $\Omega:=[0,1]^{\N^*\times\Z}$, where $[0,1]$ is endowed with Borel $\sigma$-
algebra and Lebesgue measure, and $\Omega$ with the product structure. 

For $(k,j)\in \N^*\times\Z$ and $S\subset [0,1]$, let $P_{k,j}
(S):=\prod_{(i_1,i_2)\in \N^*\times\Z}S_{i_1,i_2}$, 
where $S_{i_1,i_2}=S$ if $(i_1,i_2)=(k,j)$ and $[0,1]$ otherwise. Then we define 

\begin{equation*}
A_{k,j}^{+}:=P_{k,j}([0,2^{-1}(u_k)^{-1}]),
\end{equation*}

\begin{equation*}
A_{k,j}^{-}:=P_{k,j}([2^{-1}(u_k)^{-1},(u_k)^{-1}]),
\end{equation*}

\begin{equation*}
A_{k,j}^{(0)}:=P_{k,j}([(u_k)^{-1},1]),
\end{equation*}

the map $T$ by  $T\left(\suite{x_{k,j}}{(k,j)\in \N^*\times\Z}\right):=
\suite{x_{k,{j+1}}}{(k,j)\in \N^*\times\Z}$, and 
\begin{equation*}
\xi_k:=\chi_{A_{k,0}^+}-\chi_{A_{k,0}^-}.
\end{equation*}
\end{proof}

%%%%%%%%%%%%%%%%%%%%%%%%%%%%%%%%%%%%%%%%%%%%%%%%%%%%%%%%%%%%%%%%%%%
%%                                                               %%
%% Use the two commands below for producing your bibliography    %%
%% with bibtex, then comment again the commands and include the  %%
%% content of the .bbl file in this file below the commands.     %%
%%                                                               %%
%%%%%%%%%%%%%%%%%%%%%%%%%%%%%%%%%%%%%%%%%%%%%%%%%%%%%%%%%%%%%%%%%%%

\textbf{Acknowledgments.} The authors would like to thank both 
referees for helpful comments, and for suggesting
Remark~\ref{remark_on_thm_MP06}.

%\bibliographystyle{amsplain}
%\bibliography{bibliothese}

% add below the content of your .bbl file produced by bibtex.

%%%%%%%%%%%%%%%%%%%%%%%%%%%%%%%%%%%%%%%%%%%%%%%%%%%%%%%%%%%%%%%%%%%
%%                                                               %%
%% You may add acknowledgments (optional).                       %%
%%                                                               %%
%%%%%%%%%%%%%%%%%%%%%%%%%%%%%%%%%%%%%%%%%%%%%%%%%%%%%%%%%%%%%%%%%%%

%%%%%%%%%%%%%%%%%%%%%%%%%%%%%%%%%%%%%%%%%%%%%%%%%%%%%%%%%%%%%%%%%%%
%%                                                               %%
%% You have reached the end of your document.                    %%
%%                                                               %%
%%%%%%%%%%%%%%%%%%%%%%%%%%%%%%%%%%%%%%%%%%%%%%%%%%%%%%%%%%%%%%%%%%%

\end{document}